\documentclass{amsart}
\usepackage{hyperref}
\usepackage{mathrsfs}
\usepackage{amsmath}
\usepackage{amssymb}
\usepackage[all]{xy}
\usepackage{latexsym}
\usepackage[utf8]{inputenc}
\usepackage[T1,T5]{fontenc}

\title[$\cala$-decomposability of  the  Singer construction]{The 
$\cala$-decomposability of  the  Singer construction}
\author[Nguy{\~\ecircumflex}n H. V. H{\uhorn}ng]{Nguy{\~\ecircumflex}n H. V. 
H{\uhorn}ng}
\address{Department of Mathematics, Vietnam National University, Hanoi,
334 Nguy\~\ecircumflex n Tr\~{a}i Street, Hanoi, Vietnam}
\email{nhvhung@vnu.edu.vn}

\author[Geoffrey Powell]{Geoffrey Powell}
\address{Laboratoire angevin de recherches en mathématiques (LAREMA),
CNRS, Université d’Angers, Université Bretagne Loire, 
2 Bd lavoisier 49045 Angers Cedex 01, France.
}
\email{Geoffrey.Powell@math.cnrs.fr}


\newtheorem{thm}{Theorem}[section]
\newtheorem{prop}[thm]{Proposition}
\newtheorem{cor}[thm]{Corollary}
\newtheorem{lem}[thm]{Lemma}
\newtheorem{Theorem}{Theorem}
\theoremstyle{definition}
\newtheorem{defn}[thm]{Definition}

\theoremstyle{remark}
\newtheorem{rem}[thm]{Remark}

\newtheorem{nota}[thm]{Notation}
\newtheorem{conj}{Conjecture}
\renewcommand\labelenumi{(\roman{enumi})}
\renewcommand\theenumi\labelenumi

\newcommand{\monom}{\mathfrak{m}}
\newcommand{\rtilde}[1][s]{\widetilde{R_{#1}}}
\newcommand{\lh}{\mathsf{l}}
\newcommand{\gl}{\mathrm{GL}}
\renewcommand{\theta}{\vartheta}
\renewcommand{\ker}{\mathrm{Ker}}
\newcommand{\rmfull}{{\mathrm{full}}}
\newcommand{\full}{\mathsf{f}}
\newcommand{\Pbar}{\overline{P}}
\newcommand{\image}{\mathrm{Im}}
\newcommand{\mil}{Q}

\newcommand{\ext}{\mathrm{Ext}}

\newcommand{\dash}{\mbox{-}}

\newcommand{\unst}{\mathscr{U}}
\newcommand{\st}{\mathrm{St}}
\renewcommand{\phi}{\varphi}
\renewcommand{\epsilon}{\varepsilon}
\newcommand{\nat}{\mathbb{Z}_{+}}    
\newcommand{\zed}{\mathbb{Z}}
\newcommand{\field}{\mathbb{F}_2}  
\newcommand{\tor}{\mathrm{Tor}}               

\newcommand{\cala}{\mathscr{A}}
\newcommand{\cald}{{\mathcal{D}}}

\begin{document}

\begin{abstract}
Let  $R_s M$ denote the Singer construction on an unstable module  $M$ over the 
Steenrod algebra $\cala$ at the prime two;
$R_sM$ is canonically a subobject of $P_s\otimes M$, where $P_s = \field [x_1, 
\ldots , x_s]$ with generators of degree one and $\field$ is the field with two elements. 
Passage to $\cala$-indecomposables gives the natural transformation $R_s M 
\rightarrow \field \otimes_{\cala} (P_s \otimes M)$, 
which identifies with      the dual of 
the composition of the Singer transfer and the Lannes-Zarati homomorphism. 

The main result of the paper proves the weak generalized algebraic 
spherical class conjecture, which was proposed by the first author. 
Namely, this morphism is  trivial on elements of positive degree when $s>2$. 
The condition $s> 2$ is necessary,  as exhibited by the spherical classes of 
Hopf invariant one and those of Kervaire invariant one.
\end{abstract}

\maketitle

\section{Introduction}

The Hurewicz map 
 \begin{eqnarray}
\label{eqn:hurewicz}
 \pi_* (\Omega^\infty  
\Sigma^\infty X) 
\rightarrow 
H_* (\Omega^\infty \Sigma^\infty X; \field)   
\end{eqnarray}
from the stable homotopy groups of a pointed space $X$ to the  homology 
 with coefficients in  $\field$ (the field with two elements) of 
the infinite loop space $QX := \Omega ^\infty \Sigma^\infty X$ is of 
significant 
interest to algebraic topologists. For 
example, for $X=S^0$ (so that $ \pi_* (Q S^0)$ gives the stable homotopy of the 
sphere spectrum),   
the spherical class conjecture predicts that the image in 
positive degree consists only of the images of elements of Hopf invariant one 
and those of Kervaire invariant one (cf. \cite{H_spherical},  \cite{Curtis}). 

The first author  (cf. \cite[Conjecture 1.1]{HT}) proposed the following 
generalization 
 (which is related to a conjecture due to Peter Eccles cited in \cite{Z}). 
Write 
$Q_0X$ for the basepoint 
component of $QX$ and henceforth always take homology (and cohomology) with 
$\field$ coefficients.     

\begin{conj}[The generalized spherical class conjecture]
\label{Conjecture_SC}
Let $X$ be a pointed CW-complex. Then the Hurewicz homomorphism $
\pi_{*}(Q_0 
X) \to H_{*}(Q_0 X)$
 vanishes on classes of $\pi_* (Q_0 X)$ of 
Adams filtration greater than $2$.
\end{conj}

The main result of this paper, Theorem \ref{main_intro},  proves Conjecture \ref{weak_conjecture_intro}, the weak algebraic version of this conjecture. To explain the latter, we first recall the Lannes-Zarati homomorphism and Singer's algebraic transfer (see the Appendix for their construction).

For $s$ a natural number, let $P_s$ denote the polynomial algebra $\field[x_1, \ldots , x_s]$ on generators of degree one, equipped with the canonical unstable algebra structure over the mod $2$ Steenrod algebra $\cala$; thus $P_s$ is isomorphic to $H^* (B \zed/2^{\oplus s})$ as an unstable algebra.

The Lannes-Zarati morphism involves the Singer functors. For each $s$, there is a  Singer functor, $R_s$, which is an exact functor defined on the category of unstable $\cala$-modules such that, for an unstable module $M$, $R_s M$ is canonically a submodule of $P_s \otimes M$ (see Section \ref{sect:background} for more detail). When $M = \tilde{H}^* (X)$, the reduced cohomology of a pointed space $X$,  the unstable modules $R_\bullet \tilde{H}^* (X)$  
arise in identifying $H^* (QX) $ in terms of $H^* (X)$ as an unstable $\cala$-module (cf. the presentation of the Singer functors in \cite{K_Whitehead} for homology).

Lannes and Zarati  \cite{LZ} constructed a
natural morphism 
\[
 \field \otimes_\cala R_s M 
 \rightarrow 
 \tor_s^\cala (\field, 
\Sigma^{-s} M).
\]
The linear dual
\begin{eqnarray}
 \label{eqn:dualLZ}
\ext_{\cala}^{s}(\Sigma^{-s} M, \mathbb{F}_2) \to (\field \otimes_\cala R_s 
M)^*, 
\end{eqnarray}
is known as the Lannes-Zarati morphism.  When $M= \tilde{H}^* (X)$, it  corresponds to an associated 
graded of the Hurewicz map (\ref{eqn:hurewicz}) (see \cite{Lannes} and  \cite{Goerss}).  

The following (which is due to the first author) 
is an algebraic version of the generalized spherical class conjecture: 

\begin{conj}(cf. \cite[Conjecture 1.2]{HT}) 
\label{Conjecture_Hung}
The  Lannes-Zarati morphism (\ref{eqn:dualLZ}) vanishes in  positive stem, 
for $s>2$ and  any unstable $\cala$-module $M$.
\end{conj}

\begin{rem}     
This conjecture is hard. Current approaches rely heavily upon knowledge of the cohomology of the Steenrod algebra. It has been proved  for $M=\field$ for  $s=3,4$, and $5$ respectively in  
\cite{H_spherical} and \cite{Hung1999}, \cite{Hung2003}, and \cite{HQT}. 
That the Lannes-Zarati morphism for $M=\field$  vanishes for $s>2$ on the algebra decomposables of $\ext_{\cala}^s(\field,\field)$  was  proved  in  \cite{HP}.
   
Conjecture \ref{weak_conjecture_intro} below represents a first step, having the advantage that its formulation does not involve the cohomology of the Steenrod algebra.
\end{rem}
  
Singer's algebraic transfer  gives an algebraic approximation to the iterated $\zed/2$-transfer of stable homotopy theory \cite{S}. For an $\cala$-module $N$, it is the dual of a natural map
\[
 \psi_s : \tor_s^\cala (\field, \Sigma^{-s} N)
\rightarrow 
\field \otimes_\cala (P_s \otimes N).
\]

Taking $N$ to be an unstable module $M$, the dual of the composition of the Singer algebraic transfer with the Lannes-Zarati morphism gives a natural transformation
\[
R_s M 
\rightarrow 
\field \otimes_\cala (P_s \otimes M). 
\]

By Proposition \ref{prop:identify_LZ_Singer}, 
this identifies with the composite of the natural inclusion $R_s M \hookrightarrow P_s\otimes M$ with the passage to $\cala$-invariants. Hence the following (which is due to the first author) 
is a  weak form of Conjecture~\ref{Conjecture_Hung} (cf. \cite[Conjecture 1.5]{HT}):  
   
\begin{conj} 
[The weak generalized algebraic spherical class conjecture] 
\label{weak_conjecture_intro}
Let $M$ be an unstable $\cala$-module and $s>2$ be an integer. Every 
positive degree element of 
the Singer construction $R_sM$ is $\cala$-decomposable in 
$P_s\otimes M$. 
\end{conj}

The  main result of the paper proves this conjecture:

\begin{Theorem}\label{main_intro}
If $s>2$, the morphism  
\[
R_s M \rightarrow \field \otimes_{\cala} (P_s \otimes M) 
\]
is trivial  on elements of positive degree, for 
any unstable $\cala$-module $M$.
\end{Theorem}

\begin{rem}
\label{rem:known}
\ 
\begin{enumerate}
\item
The condition $s > 2$ is necessary: for $M = \field$, the map is 
non-zero for $s\in \{1, 2 \}$ (see \cite{H_spherical}). 
This phenomenon arises from the existence of the spherical classes of Hopf invariant one and those of Kervaire invariant one.

 \item 
Taking $M=\field$ recovers the main result of \cite{HN}, since $R_s \field$ is isomorphic to the rank $s$ Dickson algebra $D_s$ (see Section \ref{sect:background}). 
\item 
Taking $M= P_k$, for  $k$ a natural number,  recovers the main result of \cite{HN2}.
\end{enumerate}
\end{rem}

It is an important fact that the Singer functor $R_s$ takes values in  the category of unstable modules equipped with a compatible module structure over the Dickson algebra $D_s$. This allows the Singer construction on an unstable module $M$ to be written as the direct sum of $\field$-vector spaces
\[
R_s M \cong \rtilde M \oplus \st_s M
\]
where $\rtilde M \subset R_s M$  is the submodule of $D_s$-decomposables and $\st_s M \subset R_s M$ is  an explicit space of generators (see Lemma \ref{lem:rtilde:properties}). The proof of Theorem \ref{main_intro} treats the two summands separately (see Theorem \ref{thm:refined}). 

Moreover, the general formulation of Conjecture \ref{weak_conjecture_intro} allows the following reduction strategy to be applied:
\begin{enumerate}
\item 
Proposition \ref{prop:reduction_s=3} reduces to the case $s=3$, by standard properties of the Singer functors. 
\item 
Lemma \ref{lem:F(n)-reduction} reduces to studying the cases  of the free unstable modules $M= F(n)$, since these generate the category of unstable modules.
\item 
The standard embedding of $F(n)$ in $P_n$ induces an inclusion $R_3 F(n) \subset R_3 P_n$. This allows the known structure of $R_3 P_n$ as an unstable algebra over the Steenrod algebra (see  Proposition \ref{prop:invt_ring_Singer}) to be exploited.
\end{enumerate}

The proof of Theorem \ref{main_intro} builds upon some of the methods of 
Tr{\`\acircumflex}n N. Nam and the first author \cite{HN,HN2}. However, an important new ingredient is the usage of the Milnor operations $Q_0$, $Q_1$, a key element of the inductive strategy used in Section 
\ref{sect:proofs}. 

\begin{rem}
By using the reduction strategy here and the Milnor operations, it is possible to give a  short proof of the case $M= \field$, the main result of \cite{HN}.
\end{rem}

For the submodule $\rtilde[3] F(n) \subset R_3 F(n)$, the result proved gives information on the Steenrod 
operations which are required to hit elements (see Theorem \ref{thm:hyper_refined}). In particular, even in the cases treated previously in \cite{HN,HN2}, the result proved is stronger. Crucially, this additional precision is used in the inductive proof.

\bigskip

\noindent
{\bf Organization of the paper:} Background and references are provided in 
Section \ref{sect:background} and 
the results are stated in Section \ref{sect:results},  where the 
reduction arguments are explained. The first 
result, concerning the image of the Steenrod total power, is proved 
in Section \ref{sect:image_sts}. The 
 precise form of the main result is stated and proved in Section \ref{sect:proofs}. Finally, the Appendix explains why the composite of the Singer algebraic transfer and the Lannes-Zarati morphism identifies as above.

\section{Background}
\label{sect:background}

Fix $\field$ the field with two elements and write $\cala$ for the 
mod $2$ Steenrod algebra. 
Let $\cala(n)$ denote the subalgebra of $\cala$ 
generated by $\{Sq^1, \ldots , Sq^{2^n} \}$
for $n \in \nat$, where $\nat$ denotes the set of non-negative integers. 
Recall that 
 $\Phi $ denotes the doubling functor on the  category $\unst$ of unstable 
modules. (Throughout  \cite[Chapter 1]{schwartz_book} serves as our reference for the basic theory of unstable modules and unstable algebras over the Steenrod algebra.)

Let $P_s=\field[x_1,\dots,x_s]$ denote the polynomial algebra on $s$ generators 
of degree $1$, equipped with the usual actions of the  
general linear group $GL_s=GL(s,\field)$ and of the Steenrod algebra. The rank 
$s$ Dickson algebra, 
$D_s$, is the algebra of invariants
$ 
\field[x_1,\dots,x_s]^{GL_s},
$
which is an unstable $\cala$-algebra. 
The underlying algebra is polynomial
$
D_s \cong \field[c_{s,0}, \dots, c_{s,s-1}],
$ 
where $c_{s,i}$ denotes the Dickson invariant of degree $2^s-2^i$ 
(for the action of the Steenrod algebra, see \cite{W}, \cite{H}).

The  Singer functor (as used by 
Lannes and Zarati 
\cite{LZ}) is an exact functor 
\[
 R_s : \unst \rightarrow D_s \dash \unst \stackrel{\mathrm{forget}}{\longrightarrow}
\unst
\]
from unstable modules to the category $D_s\dash\unst$ of $D_s$-modules in
unstable modules or, by forgetting the $D_s$-action, to unstable modules (see \cite{LZ,Psinger}).

\begin{prop}\cite{LZ,P_viasm}
\label{prop:Singer_properties}
For $s \in \nat$  and $M$ an unstable module:
\begin{enumerate}
\item 
There are natural inclusions in $D_s\dash\unst$:
\[
R_s M 
\hookrightarrow 
D_s \otimes M 
\hookrightarrow 
P_s \otimes M. 
\]
\item 
For $t \in \nat$, these canonical inclusions induce a commutative diagram of inclusions of unstable modules
\[
\xymatrix{
R_{s+t} M 
\ar@{^(->}[r]
\ar@{^(->}[d]
&
R_s R_t M 
\ar@{^(->}[d]
\\
P_{s+t} \otimes M 
&
R_s (P_t \otimes M), 
\ar@{_(->}[l]
}
\] 
using the isomorphism $P_s \otimes P_t \cong P_{s+t}$.
\item 
There is a natural isomorphism $ \field \otimes_{D_s} R_s M \cong \Phi^s M  $ in $\unst$
and hence a canonical projection
\begin{eqnarray}
\label{eqn:proj_Phi}
 R_s M \twoheadrightarrow \Phi^s M.
\end{eqnarray}
\item
The underlying $D_s$-module of $R_sM$ is free on  $\st_s M$, the image of the $s$-iterated total Steenrod power. 
\end{enumerate}
\end{prop}

\begin{proof}
(Indications.)
By definition, $R_s M$ is  generated as a $D_s$-submodule of $P_s \otimes M$ by the image of
$\st_s : \Phi^s M \dashrightarrow P_s \otimes M$, the total Steenrod power map
\cite{LZ} (the dashed arrow indicates that this is linear but not $\cala$-linear
in general and the iterated double $\Phi^s$ ensures that the degree is 
preserved). The total Steenrod power is defined recursively using an 
isomorphism 
$P_s \cong (P_1)^{\otimes s}$ of unstable algebras, starting from 
$\st_1 (x) = \sum u^{|x|-i} \otimes Sq^i x \in P_1 \otimes M$ and setting 
$\st_s 
= \st_{1} \circ \st_{s-1}$, interpreted
in the appropriate manner. 

The map $\st_s$ takes values in $D_s \otimes M$, hence $\st_s$ is independent of the isomorphism $P_s \cong (P_1)^{\otimes s}$ used in its construction and 
$R_s M$ is contained in $D_s \otimes M$. The total Steenrod power is an $\field$-linear section of the projection $R_s M \twoheadrightarrow \Phi^s M$. That $R_s M $ is a free $D_s$-module on $\st_s M $ is a standard verification. 

From the construction, it is straightforward to see that $R_s R_t M$ is a submodule of $R_{s+t}M$  and that this is compatible with the inclusions. 
\end{proof}

If $K$ is an unstable algebra, then $R_s K$ is a sub unstable algebra of $P_s \otimes K$, containing $D_s$ as a sub unstable algebra (assuming $K$ unital). For example, in the case of the polynomial algebra $P_k$:

\begin{prop}
\cite[Theorem 1.2]{HN2}
\label{prop:invt_ring_Singer}
 For positive integers $k, s$, the unstable algebra $R_s P_k$ is isomorphic to 
the algebra of invariants 
 \[
  R_s P_k \cong P_{k+s}^{ \gl_s \bullet \mathbf{1}_k}
 \]
where $\gl_s \bullet  \mathbf{1}_k \subset \gl_{k+s}$ is the subgroup of matrices of the form 
$\left(
\begin{array}{ll}
A & * \\
0 & \mathrm{Id}_k
\end{array}
\right)$
with $A \in \gl_s$ and $\mathrm{Id}_k \in \gl_k$ the identity.
 
Moreover, there is an isomorphism of algebras 
\begin{eqnarray}
\label{eqn:RsPk}
 R_s P_k \cong \field [c_{s,i} , V_s(j) ~|~ 0\leq i < s, \  1 \leq j \leq k]
\end{eqnarray}
where $V_s (j):= \mathrm{St}_s x_j$, in particular $|V_s (j)|= 2^s$.
\end{prop}

\begin{rem}
\ 
\begin{enumerate}
\item 
This result is implicit in  \cite[Section 5.4.7]{LZ}, although the 
 underlying algebra of $R_s P_k$ is not identified explicitly.
  \item
  As an element of $P_{k+s}$, $V_s (j)$ is the Mui invariant
  $
   \prod_{y \in \field^s} (x_j + y) 
  $, 
where $y$ ranges over elements of $\field^s \subset \field^{k+s}$. 
\end{enumerate}
\end{rem}

For the proof of the  main result, it is sufficient to consider the case
$s=3$, hence the following notation is adopted:

\begin{nota}
Write $V(j)$ for $V_3(j) \in R_3 P_k$ for $1 \leq j \leq k$.
\end{nota}

Since $R_3P_k$ is a polynomial algebra on specified generators 
(\ref{eqn:RsPk}), 
it is equipped 
with a length grading:

\begin{nota}
 \label{nota:length}
Write $\lh (\monom)$ for the length of a monomial $\monom$ in  $\{ c_{s,i}
, V(j) ~|~ 0\leq i < s, \  1 \leq j \leq k \}$ (that is the length of $\monom$ 
as a
word). 
\end{nota}

\begin{defn}
\label{def:length} 
A non-zero element of $R_3 P_k$ is {\em length homogeneous} if it is the sum
of monomials $\sum_i \monom_i $ of the same length; its length is $\lh 
(\monom_i)$ for any monomial appearing.  
\end{defn}

\begin{lem}
\label{lem:length_homog_V}
A non-zero element of the sub-algebra $\field [V(j)| 1 \leq j \leq k] \subset R_3 P_k$ is length homogeneous if and only if it is concentrated in a single degree.  
\end{lem}

\begin{proof}
The generators $V(j)$ each have degree $8$. Hence, the degree of a monomial $\monom \in \field [V(j)| 1 \leq j \leq k] $ is $|\monom |= 8 \lh (\monom)$. The result follows immediately.
\end{proof}

Recall that the first two Milnor operations are $Q_0 = Sq^1$ and $Q_1= [Sq^2, 
Sq^1]$.

\begin{prop}
\label{prop:Sq_invts}
\cite{H}
The $\cala$-action on 
$R_3 P_k$ is determined by
\[
 \begin{array}{|l|l|l|l|l|l|l|l|l|}
  \hline
  Sq^0&Sq^1&Sq^2&Sq^3&Sq^4&Sq^5&Sq^6&Sq^7&Sq^8\\
  \hline\hline
 c_{3,2} & &c_{3,1}& c_{3,0} & c_{3,2}^2 &&&&\\ 
 \hline
 c_{3,1} & c_{3,0} & && c_{3,1}c_{3,2} & c_{3,0}c_{3,2} & c_{3,1}^2 &&\\
 \hline
 c_{3,0}&& &&c_{3,0}c_{3,2} & & c_{3,0}c_{3,1}& c_{3,0}^2 &\\
\hline
V(j) &&&&V(j)c_{3,2} & & V(j)c_{3,1} & V(j) c_{3,0} & V(j)^2\\
\hline
 \end{array}
\]
In particular, the Milnor operations $\mil_0, \mil_1$ act trivially on $V(j)$
and $c_{3,0}$, $Q_0 c_{3,2} = Q_1 c_{3,2}= 0$ whereas 
$Q_1 c_{3,2} =  Q_0 c_{3,1}= c_{3,0}$.
\end{prop}

The Milnor operation $\mil_i \in \cala$ satisfies  $\mil_i^2=0$ and 
the  Margolis cohomology 
groups of an unstable module $M$ are defined as 
$$H^* (M; \mil_i)
:= 
\ker \mil_i / \image \mil_i 
$$ 
with grading inherited from $M$.  Moreover, since $\mil_i \in \cala$ is 
primitive
 with respect to the Hopf algebra structure, it acts as a derivation on 
unstable 
algebras.

\begin{nota}
 Let  $\Pbar_k$ denotes the augmentation ideal of
$P_k$.
\end{nota}

\begin{lem}
\label{lem:Milnor_observation}
For $i, k \in \nat$,  
\begin{enumerate}
 \item 
 \label{item:MO_1}
 $\Pbar_k$ is $\mil_0$-acyclic, that is $\ker\mil_0 = \image \mil_0$ on 
$\Pbar_k$; 
 \item 
  \label{item:MO_2}
$\Phi P_k \subset P_k$ lies in $\ker \mil_i$ 
and induces a
surjection 
$$
\Phi P_k\cong 
\field [x_j^2 ~|~ 1 \leq j \leq k]  \twoheadrightarrow H^* (P_k; \mil_i) 
\cong 
\field [x_j^2 ~|~ 1 \leq j \leq k]/ (x_j^{2^{i+1}});
$$
\item  
\label{item:MO_3}
$\ker (\mil_i|_{\Pbar_k})= \Phi \Pbar_k + \image
\mil_i \subset 
\image \mil_0 + \image \mil_i.
$
\end{enumerate}
In particular 
$$
\ker (\mil_0|_{\Pbar_k}) + \ker (\mil_1|_{\Pbar_k}) \subset \big( \image Sq^1 + 
\image
Sq^2 \big) = \overline{\cala (1)} \ \Pbar_k, 
$$ 
where $\overline{\cala (1)}$ is the augmentation ideal of $\cala(1)$.
\end{lem}

\begin{proof}
 The first two points are standard; for instance, to 
calculate the Margolis cohomology,  by the Künneth theorem one reduces  to the 
case $k=1$, which is elementary. Point  \ref{item:MO_3}  follows from  \ref{item:MO_1} and \ref{item:MO_2}.

Since $Q_0 = Sq^1$ and $Q_1 = [Sq^2, Sq^1]$, the final statement holds.
\end{proof}

Recall that $F(n)$ denotes the
free unstable module on a generator $\iota_n$ of degree $n$, for $n \in \nat$  
(see \cite[Section 1.6]{schwartz_book});  these 
modules 
form a set
of  projective generators of $\unst$. Moreover, $F(n)$  embeds in $P_n$ as the submodule generated by the class $\prod
_{i=1}^n x_i$.

\begin{nota}
For $M$ an unstable module, 
write $\st_s M $ for the image of the linear map $\st_s : \Phi^s M 
\dashrightarrow D_s
\otimes M$.
\end{nota}

\begin{rem}
 As graded vector spaces, there is a natural isomorphism $\st_s M \cong \Phi ^s 
M$, whence $\st_s$ is considered as an exact functor. 
\end{rem}

The total Steenrod power is 
multiplicative (see \cite{StE},  for example):   if $K$ is an unstable algebra and $x, y \in K$, 
then $\st_s (xy)= \st_s (x) \st_s (y)$, where the product is formed in  $R_s K 
\subset P_s \otimes K $.

The following  is stated for $s=3$ for notational simplicity.

\begin{lem}
\label{lem:length_homog}
 For $ k$ a non-negative integer, 
\begin{enumerate} 
 \item 
 there is an embedding 
 $
  R_3 F(k) 
  \hookrightarrow 
  R_3 P_k
 $ in $D_3 \dash \unst$ and $R_3 F(k)$ identifies as the free $D_3$-submodule of $R_3 P_k$ on $\st_3 F(k) $; 
\item 
$\st_3 F(k) \subset \field [V(j)| 1 \leq j \leq k] \subset R_3 P_k$ in particular, $\st_3 F(k) $ 
has a basis of length homogeneous elements (see Definition \ref{def:length}). 
\end{enumerate}
\end{lem}

\begin{proof}
The inclusion $F(k) \hookrightarrow P_k$ induces the embedding $R_3 F(k) 
  \hookrightarrow   R_3 P_k$, since $R_3$ is exact.  Since, by definition, $\st_3 x_j = V(j)$, it is straightforward to see that $\st_3 P_k$ lies in $\field [V(j)| 1 \leq j \leq k] $ and hence  so does $\st_3 F(k)$. Length homogeneity  follows from Lemma \ref{lem:length_homog_V}.
\end{proof}

\section{Reduction arguments}
\label{sect:results}

For a  graded module $N$, let $N^{>0}$ denote the submodule of   elements of 
 positive degree. The main result of the paper   can be restated as  follows:

\begin{thm}
\label{thm:main}
For an unstable module  $M$ and an integer $s\geq 3$,
the natural map 
\[
 (R_s M) ^{>0} 
 \rightarrow 
 \field \otimes _\cala (P_s \otimes M)  
\]
induced by $R_s M \hookrightarrow  P_s \otimes M$
and passage to $\cala$-indecomposables is zero. 
\end{thm}

The following standard argument reduces to the case $s=3$:

\begin{prop}
\label{prop:reduction_s=3}
 Suppose that $
 (R_3 M) ^{>0} 
 \rightarrow 
 \field \otimes _\cala (P_3 \otimes M)  
$ is zero for each unstable module $M$, then 
\[
 (R_s M) ^{>0} 
 \rightarrow 
 \field \otimes _\cala (P_s \otimes M)  
\]
is zero for each unstable module $M$ and all $s \geq 3$. 
\end{prop}

\begin{proof}
For $s\geq 3$, by Proposition \ref{prop:Singer_properties}, there is a natural inclusion 
$R_s 
M \hookrightarrow R_3
R_{s-3} M $ and this fits into the commutative diagram 
\[
 \xymatrix{
 R_s M 
 \ar[r]
 \ar@{^(->}[d]
 &
  \field \otimes _\cala (P_s \otimes M)  
  \\
   R_3 R_{s-3} M
 \ar[r]
 &
   \field \otimes _\cala (P_3 \otimes R_{s-3}M),  
 \ar[u]
 }
\]
where the upwards arrow is induced by $R_{s-3} M \hookrightarrow
P_{s-3} \otimes M$ and the isomorphism $P_s \cong P_3 \otimes P_{s-3}$.
The lower horizontal arrow is the morphism $R_3 N \rightarrow    \field 
\otimes _\cala (P_3 \otimes N)$ with $N= R_{s-3}M$. 
By hypothesis, this is zero on positive degree elements, hence so is the upper. 
\end{proof}

\begin{nota} 
For $s\in \nat$ and $M$ an unstable module, write  $\rtilde M$ for 
$\overline{D}_s \otimes_{D_s} R_s M$, considered as an unstable module.
\end{nota}

\begin{lem}
\label{lem:rtilde:properties}
For $s \in \nat$ and $M$ an unstable module,
 \begin{enumerate}
  \item 
  $\rtilde M $ is the kernel of the canonical surjection $R_s M \twoheadrightarrow \Phi^s M \cong \field \otimes_{D_s} R_s M$  (see Proposition \ref{prop:Singer_properties} equation (\ref{eqn:proj_Phi})); 
  \item
  there is a natural isomorphism of graded vector spaces $R_s M \cong \rtilde M
\oplus \st_s M$;
  \item 
  the functor $\rtilde : \unst \rightarrow \unst$ is exact and commutes with
direct sums.
 \end{enumerate}
\end{lem}

\begin{proof}  
 The result follows by applying the  functor $- \otimes _{D_s} 
R_s 
M$
to the short exact sequence of $D_s$-modules:
$ 
 0
\rightarrow 
\overline{D}_s
\rightarrow 
D_s 
\rightarrow 
\field 
\rightarrow 
0,
$ 
since $R_s M$ is a free $D_s$-module (see Proposition \ref{prop:Singer_properties}).
\end{proof}

\begin{lem}
\label{lem:induct_include}
For $s \in \nat$ 
 and $M$ an unstable module, the canonical inclusion $R_sM
\hookrightarrow D_s \otimes M$ induces an inclusion 
 \[
  P_s \otimes _{D_s} R_s M \hookrightarrow P_s \otimes M
 \]
in the category $P_s\dash \unst$ of $P_s$-modules in $\unst$. 
 
Considered as a subobject of $P_s \otimes M$, the underlying $P_s$-module is free on $\st_s M$.
\end{lem}

\begin{proof}
The first statement is clear and the second is an immediate consequence of the fact that $
R_s M$ is free as a $D_s$-module on $\st_s M \subset P_s \otimes M$ (see Proposition \ref{prop:Singer_properties}).
\end{proof}

The natural map $R_s M \rightarrow  \field \otimes _\cala (P_s \otimes M) $
 factors 
\[
 \xymatrix{
 R_s M 
 \ar[r]
 \ar[rd]
 &
 \field \otimes _\cala (P_s \otimes_{D_s} R_s M) 
 \ar[d]
 \\
& \field \otimes _\cala (P_s \otimes M)
 }
\]
and 
restricts to a linear map 
$
\Phi^s M \cong  \st_s M \dashrightarrow  \field \otimes _\cala (P_s \otimes M). 
$

Theorem \ref{thm:main} follows (using Lemma \ref{lem:rtilde:properties}) from
the refined statement:

\begin{thm}
\label{thm:refined}
For  an unstable module  $M$ and  an integer $s>2$, 
  \begin{enumerate}
   \item 
  $ \st_s M^{>0}  \dashrightarrow  \field \otimes _\cala (P_s \otimes M)$ is 
zero;
  \item
  the composite morphism $\rtilde[s] M \hookrightarrow  R_s M \rightarrow  
\field \otimes
_\cala (P_s \otimes_{D_s} R_s M) $ is zero.
  \end{enumerate}
\end{thm}

\begin{rem}
By Proposition \ref{prop:reduction_s=3}, it suffices to treat the case $s=3$. 
For this, a more refined result is stated in Theorem \ref{thm:hyper_refined}.
\end{rem}

The proof of the  theorem  is reduced to considering the projective generators of unstable modules by the following:

\begin{lem}
\label{lem:F(n)-reduction}
 Theorem \ref{thm:refined} holds for all unstable modules $M$ if and only if it
holds for the free unstable module $F (n)$, for any  non-negative integer $n$. 
\end{lem}

\begin{proof}
 This is a formal consequence of the fact that the set of all $F(n)$, $n$ 
a non-negative integer, forms a set of projective generators of $\unst$ and 
that 
$\st_3$, $\rtilde[3]$ 
are exact and commute with direct sums, by Lemma \ref{lem:rtilde:properties}. 
\end{proof}

\section{The image of $\st_s$}
\label{sect:image_sts}

 The aim of this section is to prove Proposition \ref{prop:st3_hit} below, 
which corresponds to the first statement of  Theorem \ref{thm:refined}. 

Recall that, for $M$ an unstable module, $\st_3$ denotes the linear map 
$$\Phi^3
M \dashrightarrow D_3 \otimes M \subset P_3 \otimes M
$$
and $\st_3 M$ denotes its image. 

\begin{lem}
\label{lem:mil_trivial}
 For $M$ an unstable module, $\mil_0, \mil_1, Sq^2 $ act 
trivially on
$\st_3 M$.
\end{lem}

\begin{proof}
 It is straightforward to reduce to the universal example $x = \iota_n \in 
F(n)$, the fundamental class of the free unstable module $F (n)$ (see Section 
\ref{sect:background}). 
 Now $F(n) \hookrightarrow P_n= \field [x_1, \ldots , x_n]$ under 
the map induced by $\iota_n \mapsto \prod_{i=1}^n x_i$, so that   $\st_3 
(\iota_n) \in 
\st_3 (F(n)) \subset R_3 P_n$; here one has  $\st_3 (\iota_n) = \prod_{j=1}^n 
V(j)$.  Since $Sq^1 , Sq^2$ (and 
hence $Q_1$) act trivially on $V(j)$   by  Proposition 
\ref{prop:Sq_invts}, the result follows from the Cartan formula.
\end{proof}

Consider the subalgebra $\cala (0) \subset \cala$ generated by $\mil_0$, which 
identifies with the exterior algebra $\Lambda (\mil_0)$.

\begin{lem}
\label{lem:ker_image}
As an $\cala (0)$-module: 
$$D_3 \cong (\field [c_{3,2}], Q_0c_{3,2} =0) \otimes 
(\field [c_{3,1}, c_{3,0}], \mil_0 c_{3,1} = c_{3,0}).
 $$
Hence
\begin{enumerate}
 \item 
$H_* (D_3;\mil_0) \cong \field [c_{3,2}, c_{3,1}^2]$;
\item 
a monomial in $\{c_{3,0}, c_{3,1}, c_{3,2} \}$
of $D_3$ lies in $\ker \mil_0 $ if and only if it has the 
form $ c_{3,2}^{i_2} c_{3,1}^{2n} c_{3,0} ^{t}$ for $i_2,n, t\in \nat$; if $t>0$
then it lies in $\image \mil_0$.
\end{enumerate}
\end{lem}

\begin{proof}
Straightforward.
\end{proof}

\begin{nota}
 Index the monomial basis of $D_3$  by $I = (i_2,i_1, i_0) \in \nat^{ 3}$ with
\begin{eqnarray}
\label{eqn:monomial_basis}
c^I =  c^{i_2,i_1,i_0} := c_{3,2}^{i_2}c_{3,1}^{i_1}c_{3,0}^{i_0}.
\end{eqnarray}
\end{nota}

\begin{lem}
\label{lem:mil0_reduction}
Let $M$ be an $\cala(0)$-module and write an element $y\in 
\overline{D}_3
\otimes M$ as $y=  \sum c^I \otimes y_I $, where $c^I$ runs over the monomial 
basis (\ref{eqn:monomial_basis}) of $\overline{D}_3$.

Then $y$ lies in $ \ker Q_0$ if and only if both 
\begin{enumerate}
 \item 
$
\mil_0 y_{i_2, 2u, t+1} = y_{i_2, 2u+1, t}
$ for all $(i_2, u, t) \in \nat^{\times 3}$ and 
\item 
$
 \mil_0 y_{i_2, 2u,0} =0. 
$
\end{enumerate}

If these conditions are satisfied,  
\[
 y = \sum c^{i_2, 2u,0} \otimes y_{i_2, 2u,0} +  \sum \mil_0 (c^{i_2 , 2u+1, t} 
\otimes y_{i_2, 2u, t+1}).
\]
\end{lem}

\begin{proof}
 Clearly $Q_0 y = \sum \{ (\mil_0 c^I) \otimes y_I  + c^I \otimes \mil_0 y_I 
\}$. The result follows by identifying coefficients in the expansion in terms
of  the monomial basis of $D_3$.
\end{proof}

\begin{nota}
Denote by $\overline{\cala}$ the augmentation ideal of $\cala$  so that, for an $\cala$-module $N$, 
$\overline{\cala} N \subset N$ is the kernel of the natural projection $N \twoheadrightarrow \field \otimes_{\cala}N$. 
\end{nota}

 \begin{prop}
 \label{prop:st3_hit}
  For $M$ an unstable module,  $\st_3 M^{>0} \subset \overline{\cala} (P_3
\otimes M)$.  More precisely, 
$
 \st_3 (x) \in \image Sq^1 + \image Sq^{4|x|}   
$
for $x \in M^{>0}$. 
 \end{prop}
 
 \begin{proof}
As in Lemma \ref{lem:mil_trivial}, one reduces to the universal example 
$\iota_n 
\in F(n)$: it suffices 
to show that $\st_3 \iota_n \in \overline{\cala} (P_3 \otimes F(n))$ for each 
$0<n \in \nat$.

Lemma \ref{lem:mil_trivial} implies that $\st_3 \iota_n \in \ker \mil_0$; 
Lemma \ref{lem:mil0_reduction} shows that it suffices to consider the terms 
  \[
    c_{3,2}^{i_2} c_{3,1} ^{2u} \otimes y_{i_2,2u}
  \]
which appear in the expansion of $\st_3 \iota_n$ (for simplicity, writing 
$y_{i_2, 2u}$ for $y_{i_2, 2u, 0}$). By Lemma 
\ref{lem:mil0_reduction}, both  $y_{i_2, 2u}$ and
$c_{3,2}^{i_2} c_{3,1} ^{2u}$ 
lie in $\ker \mil_0$. 

Clearly $y_{0,0} = (Sq_0)^3 \iota_n \in F(n)$ (recall that $Sq_0 (x) := 
Sq^{|x|}x$); in the other 
cases, $c_{3,2}^{i_2} c_{3,1} ^{2u}$ has positive degree so that there exists 
$a_{i_2,2u} \in P_3$ such that $\mil_0 a_{i_2,2u} = c_{3,2}^{i_2} c_{3,1}
^{2u}$,  
by Lemma \ref{lem:Milnor_observation}. Thus $\mil_0 (a_{i_2,2u} \otimes  
y_{i_2,2u}) =   c_{3,2}^{i_2} c_{3,1} ^{2u} \otimes y_{i_2,2u}$ in $P_3 \otimes 
F(n)$.

This establishes that, for $x \in M^{>0}$, $\st_3 (x) \equiv 1 
\otimes (Sq_0)^3 (x) = (Sq_0)^3 (1 \otimes x) \mod \image Q_0$ in $P_3 \otimes 
M$. 
Since $(Sq_0)^3 (x) =  Sq^{4|x|} (Sq_0)^2(x)$, the final statement follows. 
 \end{proof}

 \begin{cor}
 \label{cor:sts_hit}
  For an integer $s \geq 3$ and  an unstable module $M$, 
$$
\st_s M^{>0} \subset
\overline{\cala} (P_s \otimes 
M). 
$$
 \end{cor}

 \begin{proof}
Use a reduction argument similar to the proof of Proposition 
\ref{prop:reduction_s=3}.
 \end{proof}

\section{Proof of the main theorem}
\label{sect:proofs}

To complete the proof of  Theorem \ref{thm:refined}, by combining Lemma 
\ref{lem:F(n)-reduction} with Proposition \ref{prop:st3_hit}, it suffices 
to show that, for each $n \in \nat$,
\[
\rtilde[3] F(n) \subset \overline{\cala} (P_3 \otimes_{D_3} R_3 F(n)).
\]
 
Recall  from Lemma \ref{lem:induct_include} that $P_3 \otimes_{D_3} R_3 M \cong P_3 \otimes 
\st_3 M$ as a $P_3$-module (forgetting the $\cala$-action). The inclusion 
$\Pbar_3 \hookrightarrow P_3$ 
of $D_3$-modules induces 
\[
 \Pbar_3 \otimes \st_3 M  \cong \Pbar_3 \otimes _{D_3}R_3 M
 \hookrightarrow 
 P_3 \otimes_{D_3} R_3 M \cong P_3 \otimes 
\st_3 M.
\]

\begin{lem}
\label{lem:basic_Milnor_reduction}
For  $M$ an unstable module, there is a natural isomorphism of 
$\cala(1)$-modules 
\[
 P_3 \otimes_{D_3} R_3 M \cong P_3 \otimes 
\st_3 M, 
\]
where $\cala(1)$ acts trivially on $\st_3 M$. 
This  restricts to an isomorphism of $\cala(1)$-modules 
$
\Pbar_3 \otimes_{D_3} R_3 M \cong \Pbar_3 \otimes 
\st_3 M.
$ 

Hence
\[
 \big(\Pbar_3 \otimes \st_3 M \big) \cap \big(\ker Q_0 + \ker Q_1 \big)
\subset \big( \image Sq^1 + \image Sq^2 \big) \subset P_3 \otimes _{D_3} R_3 M.
\]
\end{lem}

\begin{proof}
Lemma \ref{lem:mil_trivial} implies that  $\st_3M \subset R_3 M$ is a trivial 
$\cala(1)$-submodule.
The $P_3$-module structure of $ P_3 \otimes_{D_3} R_3 M$ then leads to the 
isomorphism of $P_3$-modules 
\[
 P_3 \otimes \st_3 M \rightarrow  P_3 \otimes_{D_3} R_3 M
\]
which is $\cala(1)$-linear. The case of the restriction to $\Pbar_3$ is 
straightforward.

The final statement follows from Lemma \ref{lem:Milnor_observation}.
 \end{proof}

By Lemma \ref{lem:length_homog}, $\rtilde[3] F(n) \subset R_3 F(n) \subset R_3
P_n$, where $R_3 F(n)$ is the free $D_3$-module on $\st_3 F(n)$ and 
$\st_3 F(n) \subset R_3 P_n$ has a basis given by length homogeneous polynomials in 
$\field [V(1), \ldots , V(n)]$ (cf. Definition \ref{def:length}).

\begin{lem}
\label{lem:generators}
 For  $n$ a non-negative integer, 
 $ \rtilde[3] F(n) \subset R_3 F(n)$ has a basis of elements 
of
the form 
 \begin{eqnarray}
\label{eqn:standard_basis}
 c^I \st_3 y_i
\end{eqnarray}
where $y_i \in F(n)$ ranges over a (degree) homogeneous basis  and $c^I\neq 1$ ranges over the monomial basis of $\overline{D_3}$. In particular, each $\st_3 y_i$ is length 
homogeneous.
\end{lem}

\begin{proof}
Straightforward.
\end{proof}

\begin{nota}
A {\em term} of an element of $ \rtilde[3] F(n)$ is a basis element (as in 
equation (\ref{eqn:standard_basis}) of  Lemma 
\ref{lem:generators}) which appears with non-zero coefficient. 
\end{nota}

\begin{rem}
Recall from Lemma \ref{lem:induct_include} that there are  inclusions 
\[
R_3 F(n) \hookrightarrow P_3 \otimes _{D_3} R_3 F(n)  \hookrightarrow P_3 \otimes F(n),
\]
 where the right hand terms belong to $P_3 \dash \unst$, the category of $P_3$-modules in $\unst$. This $P_3$-module structure is exploited below. Moreover, since these elements lie within the (commutative) polynomial algebra $P_3 \otimes P_n$, we allow ourselves to write multiplication by elements of $P_3$ on the right. 
\end{rem}

\begin{lem}
\label{lem:full_reduction}
Let $\monom= c^I \st_3 y \in \rtilde[3] F(n)$ where $c^I \neq 1$, $\st_3 y$ is 
length homogeneous 
and let $v \in P_3$.
If $i_1 i_2 \equiv 0\mod 2$ then $\monom v^2 \in \big( \image Sq^1 +
\image Sq^2\big)  \subset P_3 \otimes _{D_3} R_3 F(n)$.  
\end{lem}

\begin{proof}
 If $i_1 \equiv 0 \mod 2$, then $Q_0 \monom=0$ and if $i_2 \equiv 0 \mod 2$,
then $Q_1 \monom=0$, by Proposition \ref{prop:Sq_invts}. Since $\mil_i$ ($i \in 
\{0,1 \}$) is a derivation, if $\mil_i x=0$, then $\mil_i x v^2 =0$; hence, 
under the hypothesis,  
$\monom v^2 \in \ker \mil_0 +\ker \mil_1$. The result thus follows from Lemma 
\ref{lem:basic_Milnor_reduction}.
\end{proof}

This result is complemented by the following special case of  \cite[Lemma 
B]{HN}.

\begin{lem}
 \label{lem:B}
 The element $c_{3,2} c_{3,1} \in D_3$ lies in $\image Sq^1 +
\image Sq^2 \subset P_3$.
\end{lem}

Lemma \ref{lem:B} indicates the special role played by the monomial $c_{3,2} c_{3,1}$ and motivates the following decomposition of monomials in $D_3$:

\begin{lem}
\label{lem:Dickson_monomial_decomp}
Let $c^I \in D_3$ be a monomial. Then there is a unique expression 
\[
c^I = 
(c_{3,2}c_{3,1})^{2^{\sigma}-1} (c_{3,2}^{j_2} c_{3,1}^{j_1}) ^{2^{\sigma}}
c_{3,0}^{i_0}, 
\]
where $i_0, j_1,j_2, \sigma \in \nat$ and $j_1 j_2 \equiv 0 \mod 2$. 
\end{lem}

\begin{proof}
For $I=(i_0, i_1, i_2)$, we set
$
\sigma := \min \{ s(i_1), s(i_2)\},
$
where, for $\epsilon \in \{1, 2\}$, $s(i_\epsilon)$ is defined as in \cite[Definition 2.1]{HP95})
 to be the unique non-negative integer such that   
$
i_\epsilon \equiv 2^{s(i_\epsilon)} -1 \mod 2^{s(i_\epsilon)+1}
$ 
. Then the $j_1, j_2$ are determined by $i_\epsilon = (2^{\sigma} -1) + 2^{\sigma} j_\epsilon$ and, by the choice of $\sigma$, at least one of the $j_\epsilon$ is even. 
Conversely, given such an expression, $\sigma$ satisfies the above criterion. 
\end{proof}

Recall from Proposition  \ref{prop:invt_ring_Singer}  that $R_3 P_n$ is a free $D_3$-module on the polynomial algebra $\field [V(j) | 1 \leq j \leq n ]$. Hence a (degree) homogeneous element of $R_3 P_n$ can be written uniquely as
\[
\sum_I c^I p_I 
\]
for monomials $c^I \in D_3$ and $p_I \in \field [V(j) | 1 \leq j \leq n ]$ a homogeneous polynomial (which is thus length homogeneous). 

\begin{rem}
\label{rem:rtilde_Fn}
For an element of $\rtilde[3] F(n) \subset R_3 P_n$, Lemma \ref{lem:generators} implies that, in the above decomposition, $p_I = \st_3 (y_I)$ for a homogeneous element $y_I \in F(n)$.  
\end{rem}

\begin{defn}
\label{defn:fullness}
For $\monom = c^I p_I  \in R_3 P_n$ with $c^I $ a monomial in $D_3$ and $p_I\in \field [V(j) | 1 \leq j \leq n ]$ as above, 
\begin{enumerate}
 \item 
 the fullness of $\monom$, $\full(\monom) \in \nat$, is the integer $\sigma$ associated to $c^I$
 that is  given by Lemma \ref{lem:Dickson_monomial_decomp};
 \item 
 $\monom$ is full if $j_1 =0= j_2$ (i.e. $c^I = (c_{3,2}c_{3,1})^{2^f-1}c_{3,0}^{i_0}$).
\end{enumerate}   
\end{defn}

\begin{rem}
\label{rem:length_inductive_scheme}
Definition \ref{defn:fullness} applies in particular when $\monom = c^I \st_3 (y_I)$ (cf. Remark \ref{rem:rtilde_Fn}). The significance of the notion of fullness is then shown by Lemma \ref{lem:full_reduction}; this is important in the primary  induction of the proof of the main theorem below. A secondary argument uses the  length  inherited from $R_3
P_n$  (see Definition \ref{def:length} and Lemma \ref{lem:length_homog}). 
\begin{enumerate}
\item 
The length argument uses Lemma \ref{lem:exclude} below (which is inspired by 
\cite[Lemma 3.4]{HN2}), which allows monomials which satisfy 
\[
 \lh (\monom) \not \equiv 2^{\full(\monom)}-1 \mod 2^{\full({\monom})}
\]
to be treated by using an induction upon the fullness. 
\item 
The remaining case is studied by using Lemma \ref{lem:parity_case} (which is 
inspired by \cite[Lemma 3.5]{HN2}); this reduces 
to terms of smaller fullness or terms which can be treated by  Lemma 
\ref{lem:exclude}.
\end{enumerate}
\end{rem}

The following elementary observation is used:

\begin{lem}
\label{lem:2-adic_binomial}
Let  $f$ and $\ell$ be non-negative integers. Then $\ell \equiv 2^{f}-1 \mod  
2^{f}$ if and only if
  \[
  \binom {\ell}{2^i} \equiv 1 \mod 2,
 \]
for each  $0 \leq  i < f $.
\end{lem}

\begin{lem}
\label{lem:exclude}
 For $\monom = c^I \st_3 y \in \rtilde[3] F(n)$ (where $\st_3 y$ is length 
homogeneous) of length $\lh (\monom)$, and $\full
(\monom)= f>0$, suppose that, for some $0 \leq i < f$, 
 \[
  \binom {\lh (\monom)}{2^i} \equiv 0 \mod 2. 
 \]
Then, $\monom c_{3,2}^{-2^i} \in \rtilde[3]F(n)$ and, for any element $v \in
P_3$, there is an equality in $P_3 
\otimes_{D_3} R_3 F(n)$:
\[
 Sq^{4\cdot 2^i} (\monom c_{3,2}^{-2^i} v ^{2^{f}})     
 = 
 \monom v ^{2^{f}}
 + 
 \sum_l \mathfrak{s}_l v_l ^{2^{f}} 
\]
where $v_l \in P_3$ and $\mathfrak{s}_l \in R_3 F(n)$ is an element of the form
$c^{I_l} \st_3 (y_l)$ such that 
\begin{enumerate}
  \item 
$\mathfrak{s}_l \in \rtilde[3] F(n)$;
\item 
 $\full(\mathfrak{s}_l)\leq i<f $.
\end{enumerate}
\end{lem}

\begin{proof}
The fact that $\monom c_{3,2}^{-2^i} \in \rtilde[3]F(n)$ is clear. 
 The remainder of the argument uses the action of the Steenrod algebra given by Proposition \ref{prop:Sq_invts}, together with the Cartan 
formula, which gives 
\begin{eqnarray}
\label{eqn:cartan_monom}
Sq^{4\cdot 2^i} (\monom c_{3,2}^{-2^i} v ^{2^{f}})   
= 
\sum_{n'+n'' = 4 \cdot 2^i}
Sq^{n'} (\monom c_{3,2}^{-2^i}) Sq^{n''} (v^{2^f}).
\end{eqnarray}
The term $Sq^{n''} (v^{2^f})$ vanishes unless $n'' \cong 0 \mod 2^f$, say $n'' = a 2^f$,  in which case $Sq^{a 2^f} (v^{2^f}) = (Sq^a v )^{2^f}$. This means that 
$Sq^{4\cdot 2^i} (\monom c_{3,2}^{-2^i} v ^{2^{f}})$ can be written as a sum of the form $\sum_j \monom_j (v_j)^{2^f}$, where $\monom_j$ is a term of $\rtilde[3] F(n) \subset R_3 P_n$ and $v_j \in P_3$. It remains to consider the fullness of the  $\monom_j$. 

The analysis proceeds by considering the exponent of $c_{3,2}$ in the monomials  $Sq^{n'}(\monom c_{3,2}^{-2^i})$.  By Proposition \ref{prop:Sq_invts}, $Sq^4$ acts on the algebra generators of $R_3 P_n$ by multiplication by $c_{3,2}$ and $Sq^5 c_{3,1} = c_{3,0}c_{3,2}$. No other action of a Steenrod square on the algebra generators introduces new occurrences of $c_{3,2}$;  in particular, this is the case for the operations $Sq^i$, $0\leq i \leq 3$.  
 
For a term appearing in the expansion of $Sq^{4\cdot 2^i} (\monom 
c_{3,2}^{-2^i} v ^{2^{f}}) $ to have fullness greater than $i$, the action of $Sq^{n'}$ on 
 $\monom c_{3,2}^{-2^i}$ must have created $2^i$ additional factors $c_{3,2}$. Expanding using the Cartan formula, the above analysis shows that this is only possible by application of $Sq^4$ to $2^i$ algebra generators of $R_3 P_n$. In particular, such a term must  arise from the expression  $Sq^{4\cdot 2^i} (\monom 
c_{3,2}^{-2^i}) v ^{2^{f}}$
 of equation (\ref{eqn:cartan_monom}). The hypothesis on the binomial coefficient ensures that this contributes  $ \monom v ^{2^{f}}$. 

The remaining terms have fullness $\leq i$ and can be written in the required
form $ \mathfrak{s}_l v_l ^{2^{f}}$, indexing as in  the statement of the Lemma.  The fact that  $\mathfrak{s}_l\in 
\rtilde[3] F(n)$ is clear, since it is impossible to
destroy a contribution from $\overline{D}_3$.
 \end{proof}

\begin{lem}
\label{lem:parity_case}
 Let $\monom = c^I \st_3 y  \in \rtilde[3] F(n)$ with $\st_3 y$ length 
homogeneous and $\monom$ full with   $\full(\monom) = f$. 

Suppose that $\lh (\monom) \equiv 2^{f}-1 \mod 2^{f}$ and consider a Steenrod
operation $\theta \in \cala$ such that $|\theta|\leq 2^{f+1}$. 
 
 For $\mathfrak{t}$ a term of $\theta \monom$, 
\begin{enumerate}
 \item 
 if $\lh(\mathfrak{t})=\lh (\monom)$, then  $\full(\mathfrak{t}) < f$;
 \item 
 if $\lh(\mathfrak{t})> \lh (\monom)$, then $\lh (\mathfrak{t})\not  \equiv 2^f 
-1  \mod 2^f$. 
\end{enumerate}

In particular, one of the following holds:
\begin{enumerate}
  \item 
 $ \full( \mathfrak{t}) < f$;
 \item 
 $ \full( \mathfrak{t}) \geq  f$ and $\lh (\mathfrak{t})\not  \equiv 
2^{\full(\mathfrak{t})} -1  \mod 2^{\full(\mathfrak{t})}$.
 \end{enumerate}
\end{lem}

\begin{proof}
The hypothesis  that $\monom$ is full of fullness $f$ means that  $$\monom = 
(c_{3,2}c_{3,1})^{2^{f}-1} c_{3,0}^{i_0} \st_3 y.$$

In the case $\lh (\mathfrak{t})= \lh (\monom)$, $\mathfrak{t}$ arises from the
action of the operation $\theta$  upon the factor $(c_{3,2}c_{3,1})^{2^{f}-1}$ 
of $\monom$ since
Steenrod operations increase the length when operating non-trivially on the 
other 
terms (this is where fullness is used). It follows by inspection 
that 
 $\full(\mathfrak{t}) < f$.

In the remaining case, write $\lh (\monom) = 2^{f}-1 + k 2^f$, for some $k \in 
\nat$. Then, by inspection of the action given by Proposition 
\ref{prop:Sq_invts}, 
\[
 \lh (\monom) = 2^{f}-1 + k 2^{f} < \lh (\mathfrak{t}) \leq  2^{f}-1 + k 2^{f}
+ 
|\theta|/4 \leq  2^{f}-1 + k 2^{f} + 2^{f-1}.
\]
The conclusion follows by elementary arithmetic.
\end{proof}

Theorem \ref{thm:refined} is implied by the following more precise 
result. The main case of interest is when $v=1$ in the statement;    the 
inductive proof requires the  stronger statement.

\begin{thm}
 \label{thm:hyper_refined}
 Let $\monom = c^I \st_3 y \in \rtilde[3] F(n)$ be a basis element with  $\full 
(\monom)= f$ and let $v \in P_3$. 
 Then  
 \begin{enumerate}
  \item 
  if $\monom$ is full (hence $f>0$)
  \[
   \monom v^{2^{f+1}} \in \overline{\cala (f)}  \big(P_3 \otimes_{D_3} R_3 
F(n)\big);
  \]
\item 
otherwise 
\[
   \monom v^{2^{f+1}} \in \overline{\cala (f+1)} \big(P_3 \otimes_{D_3} R_3 
F(n)\big).
  \]
 \end{enumerate}
\end{thm}

\begin{proof}
 The result is proved by increasing induction upon $f$. 

The initial case of the induction for  $f=0$ (only the non-full case occurs), 
is 
established by Lemma \ref{lem:full_reduction} since, 
 by definition of fullness, the hypothesis $i_1 i_2 \equiv 0 \mod 2$ is 
satisfied. 

The initial case of the induction for the full case is for $f=1$; here one can 
write:
\[
 \monom v^4 = c_{3,2} c_{3,1} \monom' v^4
\]
and $Sq^1, Sq^2$ act trivially on $\monom' v^4$ (since $\monom'$ contains 
neither $c_{3,2}$ nor $c_{3,1}$). The result 
therefore follows from Lemma \ref{lem:B} and the Cartan formula. 

There are two inductive steps to consider. In order to present a unified proof 
of the final step, the 
induction first treats the non-full case with $\full (\monom)= f$ and then the 
full case with $\full (\monom)= f+1$.
\begin{enumerate}
 \item 
In the non-full case with $\full (\monom)= f$, one can write 
\[
 \monom v^{2^{f+1}} = (c_{3,2}^{j_2}c_{3,1}^{j_1} v^2)^{2^{f}} 
\monom_{\mathrm{full}} 
\]
where $\monom_{\mathrm{full}}$ is full with $\full (\monom_{\mathrm{full}}) = 
f$,  $j_1 + j_2 >0$ and $j_1j_2 \equiv 0 \mod 2$. Lemma \ref{lem:full_reduction}
 implies that $c_{3,2}^{j_2}c_{3,1}^{j_1} v^2 \in \big( \image Sq^1 + \image 
Sq^2 \big) \subset \Pbar_3$, say is $Sq^1 u_1 + Sq^2 u_2$, so that 
$ (c_{3,2}^{j_2}c_{3,1}^{j_1} v^2)^{2^{f}} = Sq^{2^f} u_1^{2^f} + Sq^{2^{f+1}} 
u_2 ^{2^f} $.

Hence it is sufficient to consider elements of the form 
\[
 (Sq^{2^f} u_1^{2^f} + Sq^{2^{f+1}} u_2 ^{2^f} ) \monom_{\mathrm{full}} , 
\]
where $u_1, u_2 \in \Pbar_3$. For such elements, see below.
\item 
In the full case with $\full (\monom)= f+1$, first observe that it suffices to 
prove the case for $v=1$, since
if $\monom \in \image ( Sq^1, \ldots , Sq^{2^{f+1}} )$,  the Cartan formula 
shows that 
$\monom v^{2^{f+2}}$ is also. 

Now write 
\[
 \monom = (c_{3,2}c_{3,1})^{2^{f}} \monom_{\mathrm{full}} 
\]
where $\monom_{\mathrm{full}}$ is full with $\full (\monom_{\mathrm{full}}) = 
f$. By Lemma \ref{lem:B}, 
$c_{3,2}c_{3,1}$, considered as an element of $P_3$,  lies in $\image Sq^1 + 
\image Sq^2$. Hence 
$(c_{3,2}c_{3,1})^{2^f}$ lies in   $\image Sq^{2^f} + \image Sq^{2^{f+1}}$, thus
 it is sufficient to consider elements of the form:
\[
 (Sq^{2^{f}} u_1^{2^{f}} + Sq^{2^{f+1}} u_2 ^{2^{f}} ) \monom_{\mathrm{full}}, 
\]
where $u_1, u_2 \in \Pbar_3$.  This expression has the same form as in 
the non-full case. 
\end{enumerate}

The indexing was chosen so that $\full (\monom_{\rmfull}) = f$ in both cases.
To establish the inductive steps, it is sufficient to show, for 
$u_1, u_2 \in \Pbar_3$ and $\full (\monom_\rmfull) =f$, that 
\[
 (Sq^{2^{f}} u_1^{2^{f}} + Sq^{2^{f+1}} u_2 ^{2^{f}} ) \monom_\rmfull  
 \in 
 \image (Sq^1, \ldots , Sq^{2^{f+1}}) \subset P_3 \otimes_{D_3} R_3 F(n).
\]
(Note that the appropriate subalgebra of $\cala$ is $\cala (f+1)$ in both 
cases, 
by the choice of indexing.)

Hence  consider elements of the form 
\[
(Sq^{\delta} u)^{2^f} \monom_\rmfull =   Sq^{2^{f -1  +\delta}} (u^{2^f}) 
\monom_\rmfull
\]
for $u \in \Pbar_3$ and $\delta \in \{1, 2 \}$ and $\full (\monom_\rmfull) =f$.

If $\lh (\monom_\rmfull) \not \equiv 2^f -1 \mod 2^f$, then Lemma 
\ref{lem:exclude}, 
using $Sq^{2^{i+2}}$ for the appropriate $i$,  $0 \leq i < f$, 
allows reduction to terms earlier in the inductive scheme, since 
\begin{enumerate}
 \item 
 $i+2 < f+2$,  so that $i+2 \leq f+1$ and hence  $Sq^{2^{i+2}} \in \cala 
(f+1)$; 
 \item 
 the terms $\mathfrak{s}_l v_l^{2^f}$ (in the notation of Lemma 
\ref{lem:exclude}) which occur have fullness $\full(\mathfrak{s}_l) < f$, hence 
are 
treated by the inductive hypothesis.
\end{enumerate}

Finally, if $\lh (\monom_\rmfull) \equiv 2^f -1 \mod 2^f$, consider 
\[
  Sq^{2^{f -1  +\delta}} (u^{2^f}  \monom_\rmfull ) 
\]
using the Cartan formula and Lemma \ref{lem:parity_case} to understand the 
terms 
which occur from applying a Steenrod square 
to $\monom_\rmfull$, which are of the form $Sq^j \monom_\rmfull$ with $j \leq 
2^{f+1}$. (In fact, only $Sq^j \monom_\rmfull$ with $j \in \{2^f , 2^{f+1} \}$ 
have to be considered, but this precision is not required.) 
Note that the operation $  Sq^{2^{f -1  +\delta}}$ lies in $\cala (f+1)$.

By the final statement of Lemma \ref{lem:parity_case},  the terms which arise 
either have 
fullness $<f$ or can be treated as above by using Lemma \ref{lem:exclude}. 

This completes the proof of the  inductive steps.
\end{proof}

\appendix
\section{The composite of the Lannes-Zarati morphism and the Singer algebraic  transfer}
\label{appendix}

Let  $e \in \ext^1_{\cala} (\Sigma^{-1} \field, P_1)$, where $P_1 =\field [x]$, be the  non-trivial extension class that is represented by the short exact sequence of $\cala$-modules
\[
0 \longrightarrow P_1 \longrightarrow \widehat{P}_1 
\longrightarrow\Sigma^{-1}\mathbb{F}_2 \longrightarrow 0, 
\] 
in which $\widehat{P}_1$ denotes the submodule of elements of degree 
$\geq  -1$ in the algebra $\field  [x^{\pm 1}]$ equipped with the structure of 
$\cala$-algebra extending that on $P_1$. For a positive integer $s$, the class $e^s \in 
\ext^s_{\cala} (\Sigma^{-s} \field, P_s)$ is  given by forming the $s$-fold tensor product, using the isomorphism $P_s \cong  P_1^{\otimes s}$.  

Recall that the destabilization functor $\cald$ from $\cala$-modules to 
unstable modules gives $\cald N$, the 
largest unstable quotient of the  $\cala$-module $N$ \cite{P_viasm}. 
There is a natural transformation $\cald N  \to \field \otimes_{\cala} N$ 
 of functors from $\cala$-modules to $\cala$-modules, where $ \field 
\otimes_{\cala} N$ has 
trivial $\cala$-module structure. This is obtained 
by applying $\cald$ to the quotient 
$N \to \field \otimes_{\cala} N$ and then composing with the canonical  
inclusion:
\[
\cald N  \to \cald(\field \otimes_{\cala} N) 
= 
 (\field \otimes_{\cala} 
N)^{\geq 0} 
\stackrel{\subset}{\longrightarrow} 
 \field \otimes_{\cala} N.     
\]
This  passes to  derived functors to give    
\begin{eqnarray}
\label{eqn:nat_D_Tor}
 \cald_s N \rightarrow \tor_s^{\cala} (\field, N).
\end{eqnarray}

The cap product with  $e^s$ induces 
$\cap e^s : \cald_s (\Sigma^{-s} N) \rightarrow \cald (P_s \otimes N)$ for any 
$\cala$-module $N$ and hence,  for $M$ an unstable module,
\[
 \alpha^M_s : \cald_s (\Sigma^{-s} M) \rightarrow P_s \otimes M
.\]

Lannes and Zarati \cite{LZ} used this to relate the functor $\cald_s$ to the 
Singer functors $R_s$ (see Section \ref{sect:background}). Namely, by \cite[Théorème 2.5]{LZ}, for an unstable module $M$,  $\alpha^{\Sigma M}_s$ induces an isomorphism 
\begin{eqnarray}
\label{eqn:alpha_LZ}
 \cald_s (\Sigma ^{1-s } M)
 \stackrel{\cong}{\longrightarrow} 
 \Sigma R_s M .
\end{eqnarray}
Hence, there is a   
natural morphism of $\cala$-modules 
$\Sigma  R_s M \rightarrow \tor_s^\cala (\field, 
\Sigma^{1-s} M) \cong \Sigma \tor_s^\cala (\field, 
\Sigma^{-s} M) $ and thus 
$
 \field \otimes_\cala R_s M 
 \rightarrow 
 \tor_s^\cala (\field, 
\Sigma^{-s} M).
$ 
The linear dual of this map is the Lannes-Zarati homomorphism, as explained in the Introduction.

For an $\cala$-module $N$, Singer's algebraic transfer \cite{S} is the dual of the map
\[
 \psi_s : \tor_s^\cala (\field, \Sigma^{-s} N)
\rightarrow 
\field \otimes_\cala (P_s \otimes N) 
\]
induced by the cap product with $e^s$. Hence this fits into the commutative diagram:
\begin{eqnarray*}
\label{eqn:commutative_cap}
 \xymatrix{
 \cald_s (\Sigma^{-s} N)
 \ar[r]^{\cap e^s}
 \ar[d]
 &
\cald( P_s \otimes N) 
 \ar[d]
 \\
 \tor_s^\cala (\field, \Sigma^{-s} N)
 \ar[r]_{\psi_s}
 &
 \field \otimes_\cala (P_s \otimes N)
}
\end{eqnarray*}
in which the vertical morphisms are the natural transformations (\ref{eqn:nat_D_Tor}). 
  
\begin{prop}
\label{prop:identify_LZ_Singer}
For $M$ an unstable module and $s \in \nat$, the dual of the composition of the Singer algebraic transfer with the Lannes-Zarati morphism 
is the composite 
\[
  R_s M \hookrightarrow P_s \otimes M \twoheadrightarrow \field \otimes_\cala 
(P_s \otimes M)
\]
of the canonical inclusion with the projection to $\cala$-indecomposables. 
\end{prop}  
  
\begin{proof}
This follows from the previous discussion. Namely, take $N= \Sigma M$ and consider the commutative diagram (\ref{eqn:commutative_cap}). Together with the  isomorphism $\alpha_s^{\Sigma M}$ of equation (\ref{eqn:alpha_LZ}) this yields the result after desuspension.
\end{proof}

\smallskip
\noindent
{\bf Acknowledgement.} 
This research was carried out when the first  author was a CNRS invited 
researcher at the LAREMA, Angers in the autumn of 2014. The public manuscript 
was 
prepared in the winter of 2015-16, when the first 
author visited  the Vietnam Institute for Advanced Study in Mathematics 
(VIASM), Hanoi; he would like to express his warmest thanks to the CNRS  and to
the VIASM for hospitality and for the wonderful working conditions. 

The first author was partially funded by the National Foundation for 
Science and Technology Development 
(NAFOSTED) of Vietnam under grant number 101.04-2014.19. The second  author was partially supported by the project {\em Nouvelle 
Équipe},  convention No. 2013-10203/10204 between the Région des Pays de la 
Loire and the Université d'Angers.

The authors are grateful to a referee for their careful reading of the paper. 

\providecommand{\bysame}{\leavevmode\hbox to3em{\hrulefill}\thinspace}
\providecommand{\MR}{\relax\ifhmode\unskip\space\fi MR }
\providecommand{\MRhref}[2]{%
  \href{http://www.ams.org/mathscinet-getitem?mr=#1}{#2}
}
\providecommand{\href}[2]{#2}

\end{document}